\documentclass[reqno]{amsart}
\usepackage{amssymb,latexsym}
\usepackage{amsmath,color}
\usepackage{amsthm}
\usepackage{epsfig}
\usepackage{graphicx}
\usepackage{hyperref}
\usepackage{titletoc}
\usepackage{palatino,mathpazo}
\numberwithin{equation}{section}
\newtheorem{theorem}{Theorem}[section]
\newtheorem{proposition}[theorem]{Proposition}
\newtheorem{lemma}[theorem]{Lemma}

\theoremstyle{definition}
\newtheorem{defn}{Definition}[section]

\def\XXint#1#2#3{{\setbox0=\hbox{$#1{#2#3}{\int}$}
		\vcenter{\hbox{$#2#3$}}\kern-.5\wd0}}

\def\B{\mathbb{R}^n}
\def\R{\mathbb{R}_+^{n+1}}
\def\OR{\overline{\mathbb{R}_+^{n+1}}}
\def\ou{\overline{u}}
\def\ov{\overline{v}}
\def\ow{\overline{w}}

\def\o{\overline}
\def\l{\lambda}
\def\s{(-\Delta)^s}

\def\S{\mathbb{S}^{n-1}}
\def\e{\varepsilon}
\def\r{\mathbb{R}}
\def\E{\mathcal{E}}
\def\dv{\mathrm{div}}
\def\ve{\varepsilon}
\def\D{\Delta}
\def\vp{\varphi}
\def\S{\mathcal{S}}

\begin{document}
\title[Fractional Gel'fand]{Partial regularity of stable solutions to the fractional Gel'fand-Liouville equation}

\author[A. Hyder]{Ali Hyder}
\address{\noindent Ali Hyder, Department of Mathematics, Johns Hopkins University, Krieger Hall, Baltimore, MD, 21218}
\email{ahyder4@jhu.edu}
	
\author[W. Yang]{Wen Yang}
\address{\noindent Wen ~Yang,~Wuhan Institute of Physics and Mathematics, Chinese Academy of Sciences, P.O. Box 71010, Wuhan 430071, P. R. China; Innovation Academy for Precision Measurement Science and Technology, Chinese Academy of Sciences, Wuhan 430071, P. R. China.}
\email{wyang@wipm.ac.cn}

\thanks{The first author is supported by the SNSF   Grant No.  P400P2-183866.}
\thanks{The second author is partially supported by NSFC No.11801550 and NSFC No.11871470.}

\begin{abstract}
We analyze  stable weak solutions to the fractional Gel'fand problem
\begin{equation*}
(-\Delta)^su=e^u\quad\mathrm{in}\quad \Omega\subset\B.
\end{equation*}
We prove that the dimension of the singular  set is at most $n-10s.$
\end{abstract}
\maketitle

\section{Introduction}
Let $\Omega$ be an open subset of $\B$.  We consider the following fractional Gel'fand equation
\begin{align}
\label{eq-1}  	
(-\D)^s u=e^u\quad\text{in }\quad \Omega\subset\B,\qquad n\geq1,
\end{align}
where $s\in (0,1)$,  and  $u$ satisfies $$e^u\in L^1_{\mathrm{loc}}(\Omega)\quad \mathrm{and}\quad u\in L_s(\B).$$ Here the space $L_s(\B)$  is defined by
$$L_s(\B):=\left\{u\in L^1_{\mathrm{loc}}:\|u\|_{L_s(\B)}<\infty\right\},\quad \|u\|_{L_s(\B)}:=\int_{\r^n}\frac{|u(x)|}{1+|x|^{n+2s}}dx<\infty.$$
The non-local operator $(-\Delta)^s$ is  defined  by
\begin{equation*}
\label{1.deff}
\s\vp(x)=c_{n,s}~\mbox{P.V.}\int_{\B}\frac{\vp(x)-\vp(y)}{|x-y|^{n+2s}}dy,
\end{equation*}
with $c_{n,s}:=\frac{2^{2s-1}}{\pi^{n/2}}\frac{\Gamma(\frac{n+2s}{2})}{|\Gamma(-s)|}$  being a normalizing constant.
	
We are interested in the classical question of  regularity of solutions to \eqref{eq-1}. More precisely, we aim at studying the partial regularity of stable weak solutions of \eqref{eq-1}. By a  weak solution we mean that $u$ satisfies \eqref{eq-1} in the sense of distribution, that is
\begin{equation*}
\label{1.weak}
\int_{\r^n}u (-\D)^s \phi dx=\int_{\Omega}e^u\phi dx,\quad \forall \phi\in C_c^\infty(\Omega).
\end{equation*}
We recall that a solution $u$ to \eqref{eq-1} is said to be  a stable weak solution  if  $u\in \dot H^s(\Omega)$ and
\begin{equation}
\label{1.stable}
\frac{c_{n,s}}{2}\int_{\B}\int_{\B}\frac{(\phi(x)-\phi(y))^2}{|x-y|^{n+2s}}dxdy  \geq\int_{\B}e^u\phi^2dx, \quad \forall\phi\in C_c^\infty(\Omega),
\end{equation}
where the function space $\dot{H}^s(\Omega)$ is defined by
\begin{equation*}
\label{1.defset}
\dot{H}^s(\Omega):=\left\{u\in L_{\mathrm{loc}}^2(\Omega)\mid\int_{\Omega}\int_{\Omega}\frac{|u(x)-u(y)|^2}{|x-y|^{n+2s}}dxdy<+\infty\right\}.
\end{equation*}

For equation \eqref{eq-1} in the  whole space $\B$, the study on the classification of  stable solutions and finite Morse index solutions has attracted a lot of attentions in recent decades. In the classical case, that is $s=1$, Farina \cite{f2} and Dancer - Farina \cite{df} established non-existence of stable solutions to \eqref{eq-1} for $2\leq n\leq 9$ and non-existence of finite Morse index solutions to \eqref{eq-1} for $3\leq n\leq 9$. While for the fractional case, Duong - Nguyen \cite{dn} proved that Eq.  \eqref{eq-1} has no regular stable solution for $n<10s$. Later, the authors of this paper applied  monotonicity formula to give an optimal condition on $n,s$ such that Eq. \eqref{eq-1} has no stable solutions in whole space, see \cite{hy}. In addition, the authors also showed that  weak stable solutions are smooth, provided $n<10s$.
	
A counterpart issue for equation \eqref{eq-1} in bounded domain is to analyze the extremal solution. Let us first  recall the definition of extremal solutions. Given a bounded smooth domain $\Omega\subset\B$, consider the following problem
\begin{equation*}
\begin{cases}
\Delta u+\lambda f(u)=0\quad &\mathrm{in}\quad \Omega,\\
u=0~&\mathrm{on}\quad \partial\Omega,
\end{cases}
\end{equation*}
where $f:\mathbb{R}\to\r$ be a $C^1$ function which is convex and non-negative. There exists a constant $\lambda^*>0$ such that   for each $\lambda\in[0,\lambda^*)$ there exists a unique solution $u_\lambda(x)$, which is a classical stable solution, see \cite[chapter 3]{du}. The solution $u_{\lambda^*}:=\lim_{\lambda\uparrow\lambda^*}u_\lambda$     is called the extremal solution. When $f(u)$ is the exponential nonlinearity, it is proved that if dimension $n\leq 9$, the extremal solution is always smooth in the classical setting, see \cite{du}, while in  the fractional case, Ros-Oton - Serra \cite{r1,r2} proved that the extremal solutions are regular for $n<10s$. For more general nonlinearity $f(u)$, very recently, Cabr\'e et al. \cite{cfrs} showed that  the extremal solutions are regular when $n\leq 9$ and $f(u)$ is positive, non-decreasing, convex and superlinear at $\infty$. The restriction on the dimension $n<10$ is sharp in the sense that  there are singular weak stable solutions for $n\geq10$. Thus,  a natural question is   to understand the dimension of singular sets   when $n\geq10$ ($n\geq10s$ for the fractional case). This problem has been widely studied for the Lane-Emden equation, that is the nonlinearity is given by  $f(u)=u^p$,  in various settings, e.g., estimating the singular set of stable solutions, finite Morse index solutions and stationary solutions,  see e.g. \cite{ddf,du,p1,p2,w0}. While for the exponential nonlinearity, in the classical case, the partial regularity result has been studied by Wang \cite{w1,w2} and he proved the Hausdorff dimension of the singular set for weak stable solutions is at most $n-10$. In dimension $n=3$, Da Lio \cite{dalio} showed that the dimension of singular set for stationary solution can be at most $1$.   

	\begin{defn} A point $x_0\in\Omega$ is said to be a singular point for a solution $u$ to \eqref{eq-1} if $u$ is not bounded in any small neighborhood of $x_0$. The singular set $\mathcal{S}$ is the collection of all singular points.    \end{defn}
	It follows  from the above  definition that the singular set $\mathcal{S}$ is closed in $\Omega$. Moreover, by standard regularity theory   one gets that $u$ is regular  in $\Omega\setminus\mathcal{S}$.
	
Our main result is the following:
\begin{theorem}
\label{th1.1}
Let $u$ be a stable weak solution of \eqref{eq-1}. Then the Hausdorff dimension of the singular set $\mathcal{S}$ is at most  $n-10s$.
\end{theorem}

Theorem \ref{th1.1} is a nonlocal version of Wang's \cite{w1,w2} result. In order to prove small energy regularity results (see Proposition \ref{pr3.1} ), we shall consider the following energy 
$$\mathcal{E}(u,x_0,r)=r^{2s-n}\int_{B_r(x_0)}e^udx+r^{4s-n-2}\int_{B_r^{n+1}(x_0)\cap\R}t^{1-2s}e^{\ou} dxdt,$$
where $\ou$ is the Caffarelli-Silvestre extension of $u$ in $\R $, i.e.,
\begin{equation*}
\label{1.poisson}	
\ou(X)=\int_{\mathbb{R}^n}P(X,y)u(y)dy,~X=(x,t)\in\B\times(0,+\infty),
\end{equation*}
where
$$P(X,y)=d_{n,s}\frac{t^{2s}}{|(x-y,t)|^{n+2s}},$$
and $d_{n,s}>0$ is a normalizing constant such that $\int_{\B}P(x,y)=1$, see \cite{cs}. The function  $\ou$ satisfies
\begin{equation*}
\label{1.ext}
\begin{cases}
\mathrm{div}(t^{1-2s}\nabla \ou)=0\quad &\mathrm{in}~ \R ,\\
 -\lim_{t\to0}t^{1-2s}\partial_t\ou=\kappa_s\s u=\kappa_s e^u, &\mathrm{in}~ \Omega,
\end{cases}	
\end{equation*}
where $\kappa_s=\frac{\Gamma(1-s)}{2^{2s-1}\Gamma(s)}$.

It is important to note that  each term in the energy $\mathcal{E}(u,x_0,r)$ controls the other one in a suitable way. More precisely,  due to the stability hypothesis, the second term   controls the $L^2 $ norm of $e^u$, see \eqref{2.1}, where as, by Jensen's inequality, the second term is controlled by the first one,  see Lemma \ref{le3.1}.  This interplay turned out to be very crucial in proving the energy decay estimate, see Proposition \ref{pr3.1}.   

 We remark that  the energy decay estimate does not seem to work if we only consider one of the two terms in $\mathcal{E}(u,x_0,r)$. For instance, on one hand, if we only consider the first term (as in the local case), then we lack a Harnack type inequality for non-negative fractional sub-harmonic functions.  However, in the fractional case, we do have  a Harnack type inequality involving the extension function, see Lemma \ref{lem-est-w}. This suggests to consider the send term in the definition of the energy. On the other hand, if we only consider the second term in the energy, then the $L^1$ norm of $\ov$ (as defined in \ref{def-barv}) is of the order $\sqrt\E$, which is not good enough to prove the energy decay estimate.


The article is organized as follows: In section 2 we list some preliminary results, including the conclusions presented in our previous work \cite{hy} and some known results that would be used in the current article. In section 3, we give the proof of Theorem \ref{th1.1}.

\bigskip
\begin{center}
Notations:
\end{center}
\begin{enumerate}
\item [$X=(x,t)$] \quad represent  points in $\mathbb{R}_+^{n+1}=\B\times[0,\infty)$ and $\B=\partial\R.$
\smallskip
\item [$B_r(x)$] \quad the ball centered at $x$ with radius $r$ in  $\r^{n}$, \,$B_r:=B_r(0)$.
\smallskip
\item [$B_r^{n+1}(x)$] \quad the ball centered at $x$ with radius $r$ in  $\r^{n+1}$, $B_r^{n+1}:=B_r^{n+1}(0)$.
\smallskip
\item [$D_r(x)$] \quad the intersection of $B_r^{n+1}(x)$ and $\R$, i.e., $B_r^{n+1}(x)\cap\R,$ $D_r:=D_r(0)$.\smallskip
\item [$u_+$]\quad represents the non-negative part of $u$, i.e., $u_+=\max(u,0)$.
\smallskip
\item [$C$] \quad  a generic positive constant which may change from line to line.
\end{enumerate}

\medskip
\section{Preliminary results}

In this section we present several results that will be used in next section. First, we extend the stability condition in the fractional setting. Notice that $\ou$ is well-defined as $u\in L_s(\B)$. Moreover, $t^{\frac{1-2s}{2}}\nabla\ou\in L_{\mathrm{loc}}^2(\Omega\times[0,\infty))$ whenever $u\in\dot{H}^s(\Omega)$.  We recall from \cite{hy} that the stability condition \eqref{1.stable} can be generalized to the extended function $\ou$. Precisely, if $u$ is stable in $\Omega$ then
\begin{equation}
\label{2.stable-exe}
\int_{\R}t^{1-2s}|\nabla\Phi|^2dxdt\geq
\kappa_s\int_{\B}e^u\phi^2dx,
\end{equation}
for every $\Phi\in C_c^\infty(\OR)$ satisfying $\phi(\cdot):=\Phi(\cdot,0)\in C_c^\infty(\Omega)$.  
The following Farina type estimate has been proved in   \cite[Lemma 3.4]{hy}:
\begin{lemma}[\cite{hy}]
\label{pr2.1}	
Let $u\in \dot{H}^s(\Omega)$ be a weak stable solution to \eqref{eq-1}. Given a function $\Phi\in C_c^\infty(\OR)$ be of the form $\Phi(x,t)=\phi(x)\eta(t)$ for some $\phi\in C_c^\infty(\Omega)$ and $\eta=1$ in a small neighborhood of the origin,   we have for every $\alpha\in(0,2)$
\begin{equation}
\label{2.1}
\begin{aligned}
(2-\alpha)\kappa_s\int_{\B}e^{(1+2\alpha) u}\phi^2dx
&\leq 2\int_{\R}t^{1-2s}e^{2\alpha\ou}|\nabla \Phi|^2dxdt\\
&\quad-\frac12\int_{\R}e^{2\alpha\ou}\nabla\cdot[t^{1-2s}\nabla\Phi^2]dxdt.
\end{aligned}
\end{equation}
\end{lemma}
	
%

Though the following regularity result on Morrey's space is well-known, we give a proof for convenience. We recall that a function $f $ is in the Morrey's space $ M^p(\Omega)$  if  $f\in L^1(\Omega)$, and it  satisfies  $$ \int_{\Omega\cap B_r(x_0)}|f|dx\leq Cr^{n(1-\frac1p)}\quad\text{for every }B_r(x_0)\subset\B.$$ The norm $\|f\|_{M^p(\Omega)}$ is defined to be the infimum of constants $C>0$ for which the above inequality holds.   
\begin{lemma}
\label{th2.morrey}
Let $f\in M^{\frac{n}{2s-\delta}}(B_3)$  for some $\delta>0$. We set
$$\mathcal{R}_{s}f(x):=\int_{B_{1}}\frac{1}{|x-y|^{n-2s}}|f(y)|dy.$$
Then we have
$$\|\mathcal{R}_{s}f(x)\|_{L^\infty(B_{1})}\leq C(n,s,\delta)\|f\|_{M^{\frac{n}{2s-\delta}}(B_3)}.$$
 \end{lemma}

\begin{proof}
We set
$$F(r)=\int_{B_r(x)}|f(y)|dy.$$
Then
\begin{equation}
\label{2.m.1}
\begin{aligned}
\mathcal{R}_{s}f(x)\leq~& \int_{B_2}\frac{1}{|y|^{n-2s}}|f(x-y)|dy
=\int_0^2\rho^{2s-n}F'(\rho)d\rho,\quad x\in B_1,
\end{aligned}
\end{equation}
where $\rho=|x-y|.$ We can derive from \eqref{2.m.1} and integration by parts that
\begin{equation*}
\label{2.m.2}
\begin{aligned}
\mathcal{R}_{s}f(x)\leq~&\int_0^2\rho^{2s-n}F'(\rho)d\rho
=2^{2s-n}F(2)+(n-2s)\int_0^2\rho^{2s-n-1}F(\rho)d\rho\\
\leq~&\|f\|_{M^{\frac{n}{2s-\delta}}(B_3)}2^{2s-n}2^{n+\delta-2s}+(n-2s)\|f\|_{M^{\frac{n}{2s-\delta}}(B_2)}
\int_0^2\rho^{\delta-1}d\rho\\
\leq~&C\|f\|_{M^{\frac{n}{2s-\delta}}(B_3)}.
\end{aligned}
\end{equation*}
Thus we finish the proof.
\end{proof}

For any $x_0\in B_1,$ we set
\begin{align}
\label{def-ur}
u^{\l}(x):=u(x_0+\l x)+2s\log\l,
\end{align}
The following lemma is crucial in the proof of small energy regularity estimate.
\begin{lemma}
\label{Ls}
Let $u$ be a stable solution to \eqref{eq-1} with $\Omega=B_1$.  Let $u^{\l}$ be defined in \eqref{def-ur} for some $|x_0|<1$ and $0<\l<(1-|x_0|)^{1+\frac{n}{2s}}$. Then $$\|u^{\l}_+\|_{L_s(\r^n)}\leq C\left(1+\|u_+\|_{L_s(\r^n)}\right),$$
for some $C>0$ independent of $u$.
\end{lemma}

\begin{proof}
It is easy to see that $u^\l$ is a stable solution to \eqref{eq-1} on $B_R$ with $R:=\frac{1}{\l}(1-|x_0|)$. Hence,  by \eqref{2.stable-exe}$$\int_{B_\rho}e^{u^\l}dx\leq C\rho^{n-2s}\quad \text{for }\quad 0<\rho\leq\frac R2.$$
This would imply that
$$\int_{B_{R/2}}\frac{u^\l_+(x)}{1+|x|^{n+2s}}dx\leq \int_{B_{R/2}}\frac{e^{u^\l(x)}}{1+|x|^{n+2s}}dx\leq C. $$
As $\l<1$ we get
$$u^\l_+(x)<u_+(x_0+\l x).$$
Therefore, changing the variable $x_0+\l x\mapsto y$ we obtain
\begin{align*}
\int_{B_{R/2}^c}\frac{u^\l_+(x)}{1+|x|^{n+2s}}dx &\leq \l^{-n} \left( \int_{\{|y|\leq 2\}\cap\{ |x_0-y|\geq\frac12R\l \} }+\int_{|y|>2} \right) \frac{u_+(y)}{1+\left(\frac{|x_0-y|}{\l}\right)^{n+2s}}dy\\ &\leq \frac{C}{\l^nR^{n+2s}} \int_{|y|\leq 2}u_+(y)dy
+C\l^{2s}\int_{|y|>2}\frac{u_+(y)}{|y|^{n+2s}}dy\\
&\leq C\|u_+\|_{L_s(\r^n)}.
\end{align*}
We conclude the lemma.
\end{proof}

\medskip
\section{Proof of Theorem \ref{th1.1}}
In this section, we shall present the $\varepsilon$-regularity result and the proof of Theorem \ref{th1.1}. Before we start the iteration process for the $\varepsilon$-regularity, we need a more refinement result of \cite[Lemma 3.3]{hy}.
\begin{lemma}
\label{le3.1}
Let $e^{\alpha u}\in L^1(B_1)$.  Then $t^{1-2s}e^{\alpha\ou}\in L_{\mathrm{loc}}^1(B_1\times[0,\infty)).$ Moreover,
\begin{itemize}
\item[i)] there exists $\delta=\delta(n,s)>0$ and $C=C(n,s,\|u_+\|_{L_s(\r^n)})>0$ such that   $$\|t^{1-2s}e^{\alpha \ou}\|_{L^1(D_{1/2})}\leq C\left(\|e^{\alpha u}\|_{L^1(B_1)}+\|e^{\alpha u}\|_{L^1(B_1)}^\delta\right).$$
\item[ii)] for every $\delta_0>0$ small  there exists $r_0=r_0(n,s,\delta_0)>0$ small such that $$\int_0^{r_0}\int_{B_{1/2}}t^{1-2s}e^{\alpha \ou}dxdt \leq C\left(\|e^{\alpha u}\|_{L^1(B_1)}+\|e^{\alpha u}\|_{L^1(B_1)}^{1-\delta_0}\right).$$
\end{itemize}
\end{lemma}
	
\begin{proof}
For $X=(x,t)\in D_{1/2}$ we have
\begin{equation*}
\label{3.1}
\ou(x,t)\leq C\|u_+\|_{L_s(\r^n)}+\int_{B_1}u(y)P(X,y)dy= C+\int_{B_1} g(x,t)u(y)\frac{P(X,y)}{g(x,t)}dy,
\end{equation*}
where
\begin{align}
\label{3.2}
1> g(x,t):=\int_{B_1}P(X,y)dy\geq \delta ,
\end{align}
for some positive constant $\delta $ depending on $n$ and $s$ only. Therefore, by Jensen's inequality
\begin{align*}
\int_{B_{1/2}}e^{\alpha\ou(x,t)}dx\leq~&C\int_{B_{1/2}}\int_{B_1}e^{\alpha g(x,t)u(y)}P(X,y)dydx\\
\leq~&C\int_{B_1}\max\left\{e^{\alpha u(y)},e^{\alpha \delta u(y)} \right\}\int_{B_{1/2}}P(X,y)dxdy \\
\leq ~& C\left(\int_{B_1}e^{\delta \alpha u(y)}+\int_{B_{1}}e^{\alpha u(y)}dy\right)\\
\leq~&    C\left(\|e^{\alpha u}\|_{L^1(B_1)}+\|e^{\alpha u}\|_{L^1(B_1)}^\delta\right) ,
\end{align*}
where the last inequality follows from  H\"older inequality. Integrating the above inequality with respect to $t$ on the interval $[0,\frac12]$ we obtain $i)$.
		
To to prove $ii)$, we notice that for a given  $\delta_0>0$ small we can choose  $r_0>0$ sufficiently small  such that \eqref{3.2} holds with $\delta=1-\delta_0$ for every $x\in B_{1/2}$ and $0<t\leq r_0$. Then $ii)$ follows in a similar way.
\end{proof}
	
We can simply assume that $B_1\subset\Omega$. For fixed  $0<r<1$  we decompose
$$\ou=\ov+\ow,$$
where
\begin{align}
\label{def-barv}
\ov(x,t):=C(n,s)\int_{B_r}\frac{1}{(|x-y|^2+t^2)^\frac{n-2s}{2}}e^{u(y)}dy,\quad x\in \mathbb{R}^n,
\end{align}
where $C(n,s)>0$ is a dimensional constant such that $$-\lim_{t\to0}t^{1-2s}\partial_t\ov=\kappa_se^u\quad\text{on }\quad B_r.$$
Then $\ow$ satisfies
\begin{align}
\label{eq-barw}
\begin{cases}
\mathrm{div}(t^{1-2s}\nabla\ow)=0\quad&\text{in}\quad \r^{n+1}_+,\\
\lim\limits_{t\to 0}t^{1-2s}\partial_t\ow=0\quad&\text{in}\quad B_r.
\end{cases}
\end{align}
Notice that $\ow$ is continuous up to the boundary $B_r$.
\medskip
	
Now we prove some elementary properties of the functions $\overline{v}$ and $\ow$.
\begin{lemma}
\label{lem-L^2}
Setting $v:=\ov(x,0)$  we have
$$ \|t^\frac{1-2s}{2}\ov\|_{L^2(D_r)}+ \|v\|_{L^2(B_r)}\leq C \|e^u\|_{L^1(B_r)}^\gamma \|e^u\|_{L^2(B_r)}^{1-\gamma},\quad 0<r\leq1,$$
for some $\gamma>0$ and $C>0$ independent of $u$.   \end{lemma}
	
\begin{proof}
The function $v$ can be written as a convolution, and in fact, $$v\chi_{B_r}\leq (\Gamma\chi_{B_{2r}})*(e^u\chi_{B_r}),\quad \Gamma(x):=\frac{C(n,s)}{|x|^{n-2s}},$$ where $\chi_A$ denotes the characteristic function of the set $A$. In particular, by Young's inequality, we obtain
\begin{align}
\label{est-Lp}\|v\|_{L^p(B_r)}\leq \|\Gamma\|_{L^q(B_{2r)}}\|e^u\|_{L^2(B_r)},
\end{align}
where $p,q$ verify the following conditions
$$\quad 1+\frac1p=\frac1q+\frac12,\quad 1<q<\frac{n}{n-2s}.$$
On the other hand, it is easy to see that
\begin{align}
\|v\|_{L^1(B_r)}\leq\|\Gamma\|_{L^1(B_{2r})}\|e^u\|_{L^1(B_r)}.
\end{align}
Therefore, together with the interpolation inequality  \begin{align*}
\|v\|_{L^2(B_r)}\leq \|v\|_{L^1(B_r)}^\gamma \|v\|_{L^p(B_r)}^{1-\gamma},
\end{align*}
with $\gamma,p$ satisfying
$$\frac12=\gamma+\frac{1-\gamma}{p},\quad 0<\gamma<1,$$
we obtain
\begin{align*}
\|v\|_{L^2(B_r)}\leq C \|e^u\|_{L^1(B_r)}^\gamma \|e^u\|_{L^2(B_r)}^{1-\gamma},\quad \forall r\in(0,1]. \end{align*}
Here we choose $p$ slightly bigger than $2$ in \eqref{est-Lp}, and use the fact that the $L^q$ and $L^1$ norms of $\Gamma$ are uniformly bounded in $B_r$ if $r$ stays bounded.  The lemma follows as $\ov(x,t)\leq v(x)$.
\end{proof}
	
\begin{lemma}
\label{lem-est-w}
Setting $w=\ow(x,0)$ we have for every $0<\rho<R:=(r-|x|)$ $$c_se^{w(x)}\leq \rho^{2s-n-2}\int_{D_\rho(x) }t^{1-2s}e^{\ow}dxdt\leq R^{2s-n-2}\int_{D_R(x) }t^{1-2s}e^{\ou}dxdt,\quad x\in B_r,$$
where $$c_s=\int_{D_1}t^{1-2s}dxdt.$$
\end{lemma}
\begin{proof}
We prove the lemma only for $x=0$.  From \eqref{eq-barw} we have that
\begin{equation*}
\begin{cases}
\dv(t^{1-2s}\nabla e^{\ow})=t^{1-2s}e^{\ow}|\nabla\ow|^2\geq 0\quad  &\text{in }\quad \r^{n+1}_+,\\
\lim\limits_{t\to0}t^{1-2s}\partial_te^{\ow}= 0~\quad &\text{in }\quad B_r.
\end{cases}
\end{equation*}
Therefore, for  $0<\rho<r$
\begin{equation*}
\begin{aligned}
0\leq~& \int_{ D_\rho }\dv(t^{1-2s}\nabla e^{\ow(x,t)})dxdt
=\int_{\partial D_\rho\setminus B_\rho}t^{1-2s}\partial_\nu e^{\ow(x,t)}d\sigma(X)  \\
=~& \rho^{n+1-2s}\partial_\rho\int_{\partial D_1\setminus B_1}t^{1-2s} e^{\ow(\rho x,\rho t)}d\sigma \\
\end{aligned}
\end{equation*}
This implies that $\rho^{2s-n-1}\int_{\partial D_\rho\setminus B_\rho }t^{1-2s}e^{\ow}d\sigma$ is monotone increasing with respect to $\rho$. As a consequence, we have that $\rho^{2s-n-2}\int_{ D_\rho}t^{1-2s}e^{\ow}dxdt$ is increasing in $\rho$. Hence, using that $\ow<\ou$ and $\ow$ is continuous up to the boundary $B_r$, where the former conclusion follows from the fact that $\ov>0$, we get $$c_se^{\ow(0)}\leq  \rho^{2s-n-2}\int_{D_\rho }t^{1-2s}e^{\ow}dxdt\leq  r^{2s-n-2}\int_{D_r }t^{1-2s}e^{\bar u}dxdt.$$
It completes the proof.
\end{proof}
	
We now prove the following energy decay estimate for $\E(u,x_0,r)$. For simplicity, we set $\E(u,x_0,r)$ by $\E(x_0,r)$ and first consider $x_0=0$ in the following proposition

\begin{proposition}
\label{pr3.1}  Let $u$ be a stable solution to \eqref{eq-1} with $\Omega=B_1$ for some $u\in L_s(\r^n)$. Then there exists $\e_0>0$ and  $\theta\in(0,1)$ depending only on   $n$, $s$ and $\|u_+\|_{L_s(\r^n)}$ such that  if
\begin{equation*}
\ve:=\E(0,1)\leq\e_0
\end{equation*}
then
\begin{equation*}
\E(0,\theta )\leq \frac12\E(0,1).
\end{equation*}
\end{proposition}
	
\begin{proof}	
It follows from \eqref{2.1} that
\begin{align}
\label{est-L^2u}
\int_{B_{1/2}}e^{2u}dx\leq C\ve.
\end{align}
Then writing $\ou=\ov+\ow$, where
$\ov$ is given by \eqref{def-barv}  with $r=\frac12$, we get from Lemma \ref{lem-L^2} that
\begin{align}
\label{est-v}
\|v\|_{L^2(B_{1/2})}+\|t^\frac{1-2s}{2}\o v\|_{L^2(D_{1/2})}\leq C\ve^{\frac{1+\gamma}{2}}.
\end{align}
Then using Lemma \ref{le3.1} one can find $r_1=r_1(n,s,\gamma)>0$ such that for every  $0<r\leq r_1$
\begin{equation*}
\int_{D_r}t^{1-2s}e^{2\ou(x,t)}dxdt\leq C\e^{1-\frac\gamma2 },
\end{equation*}
where $C>0$ depends on $n,s$ and $\|u_+\|_{L_s(\r^n)}$.  Together with \eqref{est-L^2u}, \eqref{est-v} and  H\"older inequality
\begin{equation*}
\int_{B_r}ve^udx+\int_{D_r}t^{1-2s}\ov e^{\ou} dxdt
\leq  C\e^{1+\frac\gamma4}.
\end{equation*}
This in turn implies that
\begin{align*}
r^{4s-n-2}\int_{D_r\cap\{\ov\geq1\}}t^{1-2s}e^{\ou}dxdt
\leq C r^{4s-n-2}\e^{1+\frac\gamma4}.
\end{align*}
Moreover, by Lemma \ref{lem-est-w}
\begin{align*}
r^{4s-n-2}\int_{D_r\cap\{\ov\leq1\}}t^{1-2s}e^{\ou}dxdt
\leq Cr^{4s-n-2}\int_{D_r }t^{1-2s}e^{\ow}dxdt\leq Cr^{2s}\ve.
\end{align*}
Combining the above two estimates
\begin{align*}
r^{4s-n-2}\int_{D_r}t^{1-2s}e^{\ou}dxdt\leq C(r^{2s}\ve+ r^{4s-n-2}\e^{1+\frac\gamma4}) . \end{align*}
In a similar way we  can also obtain
\begin{align*}
r^{2s-n}\int_{B_r} e^{u}dx\leq C(r^{2s}\ve+ r^{2s-n}\e^{1+\frac\gamma4}) .
\end{align*}
Thus, for $0<r\leq r_1$,
$$\E(0,r)\leq C(r^{2s}\ve+ r^{4s-n-2}\e^{1+\frac\gamma4}),$$
for some $C=C(n,s,\|u_+\|_{L_s(\r^n)})>0$, where we used $s<1$ and $r_1\leq1$.  Then  we first choose $\theta >0$ small enough such that $C \theta^{2s}=\frac14$,  and  later choose $\ve_0>0$ small such that $C\theta^{4s-n-2}\ve_0^\frac\gamma4=\frac14$.  Then  for   $\ve\leq\ve_0$ we obtain
$$\E(0,\theta )\leq\frac12\ve=\frac12\E(0,1).$$	
Hence, we finish the proof.
\end{proof}
	
It is not difficult to see that
\begin{equation*}	
\E(x,\frac12)\leq 2^{n+2-4s}\E(0,1),\quad \forall x\in B_{\frac12}.	
\end{equation*}
Together with the above proposition and Lemma \ref{Ls}, one can easily get the following lemma

\begin{lemma}   Let $u$ be a stable solution to \eqref{eq-1} with $\Omega=B_1$ for some $u\in L_s(\r^n)$. Then there exists $\e_0>0$ and  $\theta\in(0,1)$ depending only on   $n$, $s$ and $\|u_+\|_{L_s(\r^n)}$ such that  if
\begin{equation*}
\E(0, 1)\leq\e_0
\end{equation*}
then
\begin{equation*}
\E(x,r\theta )\leq \frac12\E(x,r),\quad\forall x\in B_{1/2} ~\ \mathrm{and}~\ 0<r\leq\frac12.
\end{equation*}
\end{lemma}
	
By an iteration argument one can show that
$$\E(x,r)\leq Cr^\alpha,\quad\forall~ 0<r\leq\frac12,$$
for some constant $C>0$ depending  only on $\theta$, where $\alpha:=-\frac{\log2}{\log\theta}>0.$ Particularly,
\begin{equation}
\label{3.moser-est}
\int_{B_r(x)}e^{u(y)}dy\leq C r^{n-2s+\alpha},\quad\forall x\in B_{\frac12} ~\ \mathrm{and}~\ 0<r\leq\frac12 .
\end{equation}
Then we decompose $u=u_1+u_2$, with
$$u_1(x)=c(n,s)\int_{B_{1/2}}\frac{1}{|x-y|^{n-2s}}e^{u(y)}dy,\quad\forall x\in B_{\frac12},$$
where $c(n,s)$ is chosen such that
$$c(n,s)\s\frac{1}{|x-y|^{n-2s}}=\delta(x-y).$$
By \eqref{3.moser-est} and Lemma \ref{th2.morrey} we conclude that $u_1(x)$ is regular in $B_{1/6} $. While $u_2$ satisfies $\s u_2=0$ in $B_{1/2}$ and it implies that $u_2$ is smooth in $B_{1/6} $. As a consequence, we get that $u$ is continuous in $B_{1/6}$.
	
We recall that if $u$ is stable in $\Omega$ then $e^u\in L^p_{\mathrm{loc}}(\Omega)$ for every $p\in[1,5)$.  Consequently, the small energy regularity results can be stated as follows:
\begin{lemma} \label{lem35}
For every $1\leq p<5$ there exists $\ve_p>0$ depending only on $n,s$ and $\|u\|_{L_s(\r^n)}$ such that if $$\int_{B_1}e^{pu}dx\leq\ve_p,$$ then $u$ is continuous on $B_{1/6}$.
\end{lemma}
	
\begin{proof}
By H\"older inequality we get that
$$ \int_{B_1}e^udx\leq C\ve_p^\frac1p.$$
Hence, by Lemma \ref{le3.1} $$\int_{D_{1/2}}t^{1-2s}e^{\ou}dxdt\leq C\ve_p^\frac\delta p,$$
for some $\delta>0$. This shows that $$\E(0,\frac12)\leq C\ve_p^\frac\delta p.$$   Then following the above arguments  we get  that $u$ is continuous on $B_{1/6}$.
\end{proof}

\begin{proof}[Proof of Theorem \ref{th1.1}.]
The problem \eqref{eq-1} is invariant under the rescaling
$$u^\lambda (x):=u(\lambda x)+2s\log\l.$$
Therefore, if
$$r^{2ps-n}\int_{B_r(x)}e^{pu}dx\leq \varepsilon_p,$$ for some $p\in[1,5)$,
then $u$ is continuous in $B_{r/6}(x)$, thanks to Lemma \ref{lem35}. Thus, if $x\in \S$ ($\S$ is the singular set) we see that for every  $r>0$
\begin{equation*}
r^{2ps-n}\int_{B_r(x)}e^{pu}dx> \delta.
\end{equation*}
Thus by the well-known Besicovitch covering lemma we have that the Hausdorff dimension of $\S$ is at most $n-2ps$ with $p\in[1,5)$. Hence, we conclude that the Hausdorff dimension of $\S$  is at most $n-10s$ and it completes the proof.
 \end{proof}

	%

	%
	%
	%
	%
	
\end{document}